\documentclass[12pt,a4paper]{article}
\usepackage{bbm}
\usepackage{stmaryrd}
\usepackage{}
\usepackage{bbding}
\usepackage{mathrsfs}
\usepackage{amssymb}
\usepackage{amsfonts}
\usepackage{amsmath}
\usepackage{cases}
\usepackage{amsthm}
\usepackage[square, comma, sort&compress, numbers]{natbib}

\usepackage[top=1in,bottom=1in,left=1.20in,right=1.20in]{geometry}

\newtheorem{theorem}{Theorem}[section]
\newtheorem{lemma}{Lemma}[section]

\newtheorem{proposition}[lemma]{Proposition}
\newtheorem{remark}{Remark}[section]

\numberwithin{equation}{section}
\theoremstyle{plain}

\newcommand{\me}{\mathrm{e}}
\newcommand{\mi}{\mathrm{i}}

\makeatletter

\newcommand{\Rmnum}[1]{\expandafter\@slowromancap\romannumeral #1@}
\makeatother

\DeclareMathOperator{\meas}{meas}

\begin{document}
\title{Quasi-Periodic solutions of Two Dimensional  Completely Resonant Reversible Schr\"{o}dinger $d$ Systems\thanks{Shuaishuai Xue acknowledges the support from The National Natural Science Foundation of China (Grant No. 12001275).}}
\author{  Yingnan Sun\footnote{School of mathematics, Nanjing University of Aeronautics and Astronautics,
Nanjing 211106, P.R. China(sunyingnan@nuaa.edu.cn).},Shuaishuai Xue\footnote{Corresponding author:School of Mathematics, Nanjing Audit University, Nanjing 211815, P.R.China(ssx@nau.edu.cn).}  }
\date{}
\maketitle
\begin{abstract}
We introduce an  abstract KAM (Kolmogorov-Arnold-Moser) theorem for infinite dimensional
reversible Schr\"odinger systems.  Using this KAM theorem together with partial Birkhoff normal form method, we find
 the existence of
 quasi-periodic solutions for a class of completely resonant reversible  coupled nonlinear Schr\"{o}dinger $d$ systems on two dimensional torus.
\end{abstract}

\textbf{Keywords.} KAM, Schr\"{o}dinger system, quasi-periodic solutions, reversible vector field, Birkhoff normal form.

\textbf{2010 Mathematics Subject Classification:} 37K55, 35B15.

\section[]{Introduction and Main Result}
\hspace{1.5em}
Consider the following $2-$dimensional nonlinear Schr\"{o}dinger $d$ systems:
 \begin{equation}\label{CNLS1}
	\begin{cases}
		( \mi \partial_t  -\Delta   ) u_1+|u_1|^2u_1+\partial_{\bar{u}_1}G_1(|u_1|^2,\cdots,|u_d|^2)=0,\quad                                  \\
		\vdots
		\\
		( \mi \partial_t  -\Delta  ) u_d+|u_d|^2u_d+\partial_{\bar{u}_d}G_d(|u_1|^2,\cdots,|u_d|^2)=0,\,x\in\mathbb{T}^2:=\mathbb{R}^2/{2\pi \mathbb{Z}^2},
	\end{cases}
\end{equation}
with periodic boundary conditions.
$\Delta$ is the Laplacian operator with respect to $x.$
 $G_h=o(|u|^4),\,h=1,2,\cdots,d$ are real analytic functions
near $(u_1,\cdots,u_d)=(0,\cdots,0).$
For general coupled nonlinearities $G_1\neq G_2\neq \cdots \neq G_d$,
equation \eqref{CNLS1} is no longer Hamiltonian but can be viewed an infinite dimensional reversible system with respect to the  involution
$S_0 : (u_1(x),\cdots,u_d(x)) \rightarrow  (\bar{u}_1(-x),\cdots,\bar{u}_d(-x)). $
Our aim is to obtain  quasi-periodic solutions of equation \eqref{CNLS1} by reversible  KAM theory.

Starting from the initial works by Kuksin \cite{Kuksin87} and Wayne\cite{Wayne90}, KAM-type iteration to construct quasi-periodic solutions for  the one dimensional partial differential equations (PDEs)   has been well-developed, see also \cite{ Wayne90, Kuksin96, Poschel96a, Chierchia00}.  The second Melnikov conditions which may not hold true for higher dimensional  spaces are necessary for KAM iteration.
Craig-Wayne and Bourgain (CWB) method do not need the second Melnikov conditions to obtain quasi-periodic solutions. Bourgain\cite{Bourgain98, Bourgain94} first proved the existence of quasi-periodic solutions for higher dimensional nonlinear Schr\"{o}dinger equations (NLS). As to further developments of the method,
see \cite{Bourgain05,WangW16,Berti12,Berti13b}.

The CWB iteration can not help us studying   dynamics information, such as the linear stability or zero Lyapunov exponents  of solutions. For understanding dynamics  of solutions,
 KAM iteration  for  higher dimensional PDEs  appeared later. Eliasson and Kuksin\cite{Eliasson10} studied a class of  higher dimensional NLS. To handle infinitely many resonances at each KAM step,  they introduced   T\"{o}plitz-Lipschitz property of perturbation. By developing the ideas of \cite{Eliasson10} and constructing suitable
partial Birkhoff normal form, Geng, Xu and You\cite{Geng11}  and Procesi and Procesi\cite{Procesi15} studied more difficult
completely resonant NLS on $\mathbb{T}^2$ and $\mathbb{T}^d$, respectively.

KAM theorem for equation systems has also been constructed recently.
Gao and Zhang\cite{Gao11} applied KAM theorem for KdV system.
 Gr\'ebert and  Da Rocha\cite{Grebert}  proved the existence of
quasi-periodic solutions for one  dimensional NLS system.
Shi and Xu\cite{Shi} constructed quasi-periodic solutions to a class of beam system.

In is paper, we introduce a   KAM theorem for infinite dimensional reversible systems and applies it to
equation \eqref{CNLS1}.
Compared with  previous results in \cite{Geng21} also for coupled NLS, the proof of the KAM theorem is much simpler while the KAM theorem    is applicable to  coupled system of $d$ equations.
We find that the normal   frequencies  of  \eqref{CNLS1} satisfy a gap condition
$$\Omega_{hn} - \Omega_{h'n} = \sum _{i\in\mathcal{I}} \frac{1}{2\pi^2} (\xi_{hi}-\xi_{h'i})\neq 0,\, h=1,\cdots,d
$$
 which can help us eliminate the resonance caused by the coupling of equations. Using this condition,  coupled system of   more complex equations, such as  NLS in \cite{Procesi15,X}, can be studied by KAM iteration directly. For example, NLS in \cite{X}
 $$	( \mi \partial_t  -\Delta   ) u+|u|^{2p}u+H(x,u,\bar{u})=0,
 $$
has normal   frequencies  $$
\Omega_n= \varepsilon^{-3p}|n|^2+\frac{p+1}{(2\pi)^{2p}}A_p(\xi)
$$
which can provide gap condition by moving $\xi$.

Now we state our main theorem.
Let  $\mathcal{I}=\{i^{(1)}, i^{(2)},\cdots,i^{(b)}\}\subset \mathbb{Z}^2$ be an admissible set of sites of Fourier modes in \cite{Geng11}. Under periodic boundary conditions,
 we denote   the
eigenvalues of  $-\Delta$ by  $\lambda_{hn},$ with $ h=1,\cdots,d,\,n\in \mathbb{Z}^2$   respectively, satisfying
$$\lambda_{hi^{(a)}}=|i^{(a)}|^2 ,\,\,1\leq a\leq b, \quad
 \lambda_{hn}=|n|^2 ,\,\,n\notin \mathcal{I},$$
and the corresponding eigenfunctions $\phi_i(x)=\frac{1}{\sqrt{(2\pi)^{2}}}\me^{\mi \langle i,x\rangle}.$
Assume the subset
$$\xi=(\xi_1,\cdots, \xi_d )=(\xi_{hi^{(a)}}, a=1,\cdots,b,h=1,\cdots,d )\in \mathcal{O}\subset \mathbb{R}^{db}$$ is a bounded parameters set having positive Lebesgue measure and
$$ |\xi_{h}-\xi_{h'}|_1>C ,\, h\neq h'.$$
Then we have the following main result.
\begin{theorem}\label{mainresult}
For   any $0<\gamma\ll1,$
there exists a Cantor subset $\mathcal{O}_\gamma \subset \mathcal{O}$  with $\meas(\mathcal{O}\setminus \mathcal{O}_\gamma)=O(\gamma^{\frac {1}{4}})$,
such that for any $\xi={(\xi_{1i^{(1)}} ,...\xi_{di^{(b)}})}\in \mathcal{O}_\gamma ,$ equation \eqref{CNLS1}   possesses a  small amplitude analytic quasi-periodic solution of the form
 \begin{equation}\label{sol}
                       \begin{cases}
                       u_1(t,x)=\displaystyle\sum_{a=1}^b \sqrt{\frac{1}{4\pi^2}\xi_{1i^{(a)}}} e^{{\rm i}\omega_{1i^{(a)}}t} e^{{\rm i}\langle i^{(a)},x\rangle}+O(|\xi_1|^{\frac{3}{2}})                              \\
                       \vdots
                       \\
                       u_d(t,x)=\displaystyle\sum_{a=1}^b \sqrt{\frac{1}{4\pi^2}\xi_{di^{(a)}}} e^{{\rm i}\omega_{di^{(a)}}t} e^{{\rm i}\langle i^{(a)},x\rangle}+O(|\xi_d|^{\frac{3}{2}})           .
                       \end{cases}
\end{equation}
where $ \omega_{hi^{(a)}}=|i^{(a)}|^2+  O(|\xi_h|),\, h=1,\cdots, d$.
\end{theorem}

\begin{remark}
	We can choose different admissible sets for each equation in \eqref{CNLS1} because the  resonance caused by the coupling of equations disappears. But we need more  notations to describe different admissible sets, frequencies and normal forms, we choose the same set for simplicity.
\end{remark}

The rest of the paper is organized as follows.
In Section 2, we give the definitions of weighted norms for functions and vector fields.
An  abstract   KAM theorem (Theorem \ref{KAM}) for infinite dimensional  reversible systems is presented in Section 3.
In Section 4, we use the KAM theorem to prove Theorem \ref{mainresult}.

\section[]{Preliminary}
\hspace{1.5em}
For the sake of completeness, we first introduce some definitions and notations.

 Let $\mathcal{I}\subset \mathbb{Z}^2$ be a  finite subset  and $\rho> 0$, we introduce the Banach space $\ell^{\rho}_{\mathcal{I}}$ of all complex sequences
$z_h=(z_{hn})_{n\in \mathbb{Z}^2\setminus \mathcal{I}}$ with weighted norm
$$\|z_h\|_\rho:=\sum\limits_{n\in\mathbb{Z}^2\setminus \mathcal{I}}\me^{|n|\rho}|z_{hn}|<\infty,\, |n|=\sqrt{|n_1|^2+|n_2|^2},\,h=1,\cdots,d.$$


Given a finite admissible subsets of $\mathbb{Z}^2:$ $ \mathcal{I}=\{i^{(1)}, i^{(2)},\cdots,i^{(b)}\}$ .
Denote $\mathbb{Z}^2_1:=\mathbb{Z}^2\setminus \mathcal{I} $ and $\ell^{\rho}_{1}:=\ell^{\rho}_{\mathcal{I}}.$
Consider the  phase space $$\mathscr{P}_{\rho}:=(\mathbb{T}^{b}\times\mathbb{R}^b\times\ell^{\rho}_{1}\times   \ell^{\rho}_{1} )^d\ni y:=(\theta_1,I_1, z_1, \bar{z}_1,   \cdots, \theta_d,I_d, z_d, \bar{z}_d     ).$$
We introduce a complex neighborhood
\begin{equation*}
  \begin{split}
  D_\rho(r,s)=\{y:|\textmd{Im} \theta_h|<r,  |I_h|<s,\|z_h\|_{\rho}<s,  \|\bar{z}_h\|_{\rho}<s, h=1,\cdots, d\}
 \end{split}
\end{equation*}
of $\mathcal{T}^{db}_0:=$ $$\mathbb{T}^{b}\times \{I_1=0\}\times\{z_1=0\} \times\{\bar{z}_1=0\}\times\cdots \times \mathbb{T}^{b}\times \{I_d=0\}\times\{z_d=0\} \times\{\bar{z}_d=0\},$$
where $|\cdot|$ is the sup-norm for vectors.

Suppose $\mathcal{O} \subset \mathbb{R}^{db}$ is a compact parameter subset. A function $f:D_\rho(r,s)\times \mathcal{O} \rightarrow \mathbb{C}$   is real analytic in
$y$ and $C^{4}_W$ (i.e., $C^{4}-$smooth in the sense of Whitney) in $\xi\in \mathcal{O}$  and has
Taylor-Fourier series expansion
\begin{equation*}
  \begin{split}
        f(y ; \xi)
            =&\sum\limits_{k\in \mathbb{Z}^{db}, l\in \mathbb{N}^{db},\alpha, \beta\in \mathbb{N}^{d\mathbb{Z}^2_1}}f_{kl\alpha \beta,\tilde{\alpha}\tilde{\beta}}(\xi)\me^{\mi \langle k, \theta\rangle}I^l z^{\alpha}\bar{z}^{\beta},
   \end{split}
\end{equation*}
where $\langle k,\theta\rangle= \sum\limits^d_{h=1} \sum\limits^b_{a=1}k_{ha}\theta_{ha},$ $I^{l}=\prod^d\limits_{h=1}\prod^b\limits_{a=1}I^{l_{ha}}_{ha}$ and $z^{\alpha}\bar{z}^{\beta}=\prod^d\limits_{h=1}\prod\limits_{j\in\mathbb{Z}^2_1}z^{\alpha_{hj}}_{hj}\bar{z}^{\beta_{hj}}_{hj}.$  $\alpha, \beta$ have only finitely many nonzero components. We define the weighted norm of $f$ as follows
\begin{equation*}
  \begin{split}
        \|f\|_{D_\rho(r,s)\times\mathcal {O}}
            =&\sup\limits_{\|z_h\|_\rho<s,\|\bar{z}_h\|_\rho<s\atop h=1,\cdots, d} \sum\limits_{k,l,\alpha, \beta}|f_{kl\alpha \beta,\tilde{\alpha}\tilde{\beta}}|_{\mathcal{O}}\me^{|k|r}s^{|l|}|z^{\alpha}||\bar{z}^{\beta}|.
   \end{split}
\end{equation*}
where $|f_{kl\alpha \beta,  \tilde{\alpha}\tilde{\beta}}|_{\mathcal{O}}=\sup\limits_{\xi\in \mathcal{O}}\sum\limits_{0\leq i\leq4}|\partial^i_\xi f_{kl\alpha \beta}|.$

Let
\begin{equation*}
z^\varrho_{hn}=
         \begin{cases}
           z_{hn},\quad                                      &  \varrho=+,\\
          \bar{z}_{hn},\quad           &   \varrho=- .
         \end{cases} \ h=1,\cdots,d.
\end{equation*}

Consider a vector field $X(y), y\in D_\rho(r,s):$
\begin{equation}\label{pdoperator}
        X(y)=        \sum\limits_{\textsf{v}\in \mathscr{V}}X^{(\textsf{v})}(y)\frac{\partial}{\partial \textsf{v}}= X^{(\theta)}(y)\frac{\partial}{\partial \theta}+X^{(I)}(y)\frac{\partial}{\partial I}
        + X^{(z)}(y)\frac{\partial}{\partial z}+ X^{(\bar{z})}(y)\frac{\partial}{\partial \bar{z}}
\end{equation}
where
 $\mathscr{V}=\{\theta_{ha},I_{ha},\,z_{hn}, \bar{z}_{hn}: a=1,\cdots,b;n\in\mathbb{Z}^2_1;h=1\cdots,d\}.$
Suppose $X$ is  real analytic  in $y$ and depends $C^{4}_W$ smoothly on parameters $\xi\in \mathcal {O},$
 define the weighted norm of $X$ as follows\footnote{The norm of vector valued function $G:D_\rho(r,s)\times\mathcal {O}\rightarrow \mathbb{C}^b$, $b<\infty,$ is defined as $\|G\|_{D_\rho(r,s)\times\mathcal {O}}=\sum\limits^b_{a=1}\|G_a\|_{D_\rho(r,s)\times\mathcal {O}}$.} $\|X\|_{s;D_\rho(r,s)\times\mathcal {O}}= $
\begin{equation*}
\sum^d\limits_{h=1 }  ( \|X^{(\theta_h)}  \|_{D_\rho(r,s)\times\mathcal {O}}+\frac{1}{s}\|X^{(I_h)}\|_{D_\rho(r,s)\times\mathcal {O}}
+\frac{1}{s}\sum\limits_{\varrho=\pm}\sum\limits_{n\in \mathbb{Z}^2_1}\me^{|n|\rho}\|X^{(z^\varrho_{hn})}\|_{D_\rho(r,s)\times\mathcal {O}}).
\end{equation*}
Denote the Lie bracket of two vector fields $X$ and $Y$ by
  $[X,Y]=Y X-X Y.$

\section[]{A  KAM Theorem for Infinite Dimensional Reversible Systems}
\hspace{1.5em}
In this section, we give an abstract KAM theorem for infinite dimensional reversible systems.
Recall the   admissible set $\mathcal{I}$ which defined in \cite{Geng11}. Let $\mathcal{L}_1$ be the subset of $\mathbb{Z}^2_1$ with the following property:
for each $n\in \mathcal{L}_1$, there exists a unique triplet $(i, j,m)$ with $m\in \mathbb{Z}^2_1$
$i, j\in \mathcal{I}$ such that
$$i-j+n-m =0,\,\,|i|^2-|j|^2 +|n|^2-|m|^2= 0.$$
In this case, we say that $(n,m)$ is a resonant pair of the first type. $\mathcal{L}_1$ consists  of  resonant
pairs of the first type.

Let $\mathcal{L}_2$ be the subset of $\mathbb{Z}^2_1$ with the following property:
for each $n\in \mathcal{L}_2$, there exists a unique triplet $(i, j,m)$ with $m\in \mathbb{Z}^2_1$
$i, j\in \mathcal{I}$ such that
$$-i-j+n+m =0,\,\,-|i|^2-|j|^2 +|n|^2+|m|^2= 0.$$
In this case, we say that $(n,m)$ is a resonant pair of the second type. $\mathcal{L}_2$ is composed of resonant
pairs of the second type. $\mathcal{L}_2$ consists  of finitely
many resonant pairs of the second type. We have $\mathcal{L}_1\cap \mathcal{L}_2=\emptyset$ and denote
 $\mathbb{Z}^2_2:=\mathbb{Z}^2\setminus (\mathcal{I}\cup \mathcal{L}_1\cup \mathcal{L}_2 ).$

Given an involution $S:(\theta,I, z,\bar{z})\mapsto(-\theta,I, \bar{z},z).$
We consider a family of  $S-$reversible vector fields
 $X_0=N+\mathcal{A} =\sum_{h=1}^d (N_h+\mathcal{A}_h)$ with
\begin{equation}\label{VFN}
\begin{split}
  N_h=&\sum_{{i}\in \mathcal{I}}\omega_{hi}\frac{\partial}{\partial \theta_{hi}}+\sum_{n\in \mathbb{Z}^2_1}(\mi\Omega_{hn} z_{hn}\frac{\partial}{\partial z_{hn}}
  -\mi \Omega_{hn}\bar{z}_{hn}\frac{\partial}{\partial \bar{z}_{hn}})\\
  +&\sum_{n\in \mathcal{L}_1}(\mi\omega_{hi} z_{hn}\frac{\partial}{\partial z_{hn}}
  -\mi \omega_i\bar{z}_{hn}\frac{\partial}{\partial \bar{z}_{hn}}) - \sum_{n\in \mathcal{L}_2}(\mi\omega_{hi} z_{hn}\frac{\partial}{\partial z_{hn}}
  -\mi \omega_{hi}\bar{z}_{hn}\frac{\partial}{\partial \bar{z}_{hn}})
 \end{split}
\end{equation}
\begin{equation}\label{VFA}
  \begin{split}
  \mathcal{A}_h=&\sum_{n\in \mathcal{L}_1}(\mi A_{hn}z_{hn}\frac{\partial}{\partial z_{hm}}
  -\mi A_{hn}\bar{z}_{hn}\frac{\partial}{\partial \bar{z}_{hm}}) \\
  +&\sum_{n\in \mathcal{L}_2}(\mi A_{hn}z_{hn}\frac{\partial}{\partial \bar{z}_{hm}}
 -\mi A_{hn}\bar{z}_{hn}\frac{\partial}{\partial z_{hm}}),\quad \omega_{hi}, \Omega_{hn},A_{hn}\in \mathbb{R}.\\
   \end{split}
\end{equation}
which corresponds to an invariant torus on the phase space.

Consider now the perturbed $S-$reversible vector field
\begin{equation}\label{perbVF}
    X=X^0+P=N+\mathcal{A}+P(y;\xi).
\end{equation}
We will prove that, for typical  (in the sense of Lebesgue measure) $\xi\in \mathcal{O}$, the vector fields \eqref{perbVF} still admit
 invariant tori for sufficiently small $P.$   For this purpose, we need the following six assumptions:
 \begin{description}
   \item[(A1) Non-degeneracy: ]  The map $\xi\mapsto\omega(\xi)$ is a $C^{4}_W$ diffeomorphism between $\mathcal{O}$ and its image.
   \item[(A2) Asymptotics  of normal frequencies:]
   \begin{equation}\label{AsyNF}
   \Omega_{hn}=\varepsilon^{-4}(|n|^2 )+\Omega^0_{hn},\,\,n\in \mathbb{Z}^2_1.
   \end{equation}
where    $\Omega^0_{hn}\in C^{4}_W(\mathcal{O})$ with $C^{4}_W-$norm bounded by a   positive constant $L.$

   \item[(A2*)    Gap condition of normal frequencies:] There exist $\gamma>0,    $ such that

\begin{equation}\label{Mel14}
	|\Omega^0_{hn}-\Omega^0_{h'n}|>\gamma,\, h\neq h',\, n\in \mathbb{Z}^2_1.
\end{equation}
 \item[(A3) Non-resonance conditions: ]
Denote

\begin{equation*}
\begin{cases}
M_{hn}=\Omega_{hn} ,\,&n\in  \mathbb{Z}^2_2, \\

M_{hn}=\left(\begin{matrix} \Omega_{hn}+\omega_{hi}&A_{hn} \\
                     A_{hm}&\Omega_{hm}+\omega_{hj}  \\

                            \end{matrix}\right),\, & n \in\mathcal{L}_1, \\
M_{hn}=\left(\begin{matrix} \Omega_{hn}-\omega_{hi}&A_{hn} \\
                     -A_{hm}&-\Omega_{hm} +\omega_{hj}
                            \end{matrix}\right) ,\,&n \in\mathcal{L}_2.
                            \end{cases}
\end{equation*}

Suppose $ A_{hn}\in C^{4}_W(\mathcal{O})$ and there exist $ \tau>0, $ such that
\begin{equation}\label{Mel1}
   |\langle k,\omega\rangle|\geq\frac{\gamma}{|k|^\tau},\,\, k  \neq 0,
\end{equation}
\begin{equation}\label{Mel2}
  |\langle k,\omega\rangle I_{\varphi(n)}\pm M_{hn}|\geq\frac{\gamma}{|k|^\tau},
\end{equation}
\begin{equation}\label{Mel13}
  |\det(\langle k,\omega\rangle I_{{\varphi(n)}{\varphi(n')}}\pm M_{hn}^T\otimes I_{\varphi(n')}\pm I_{\varphi(n)}\otimes M_{h'n'})|\geq\frac{\gamma}{|k|^\tau},\,\, k  \neq 0
\end{equation}

where $I_a$ is $a\times a$ identity matrix, ${\varphi(n)} (resp. {\varphi(n')})$  denotes the dimension of $M_{hn}(resp. M_{hn'})$. $\det(\cdot)$, $\otimes $  and $(\cdot)^T$ denotes the determinant, the tensor product and the  transpose of matrices, respectively.

 \item[(A4) Regularity: ] $\mathcal{A}+P$ is real analytic in $y$ and $C^{4}_W-$smooth in $\xi$. Moreover,
 $\|\mathcal{A}\|_{s;D(r,s)\times\mathcal{O}}<1,$ $\varepsilon_0:=\|P\|_{s;D(r,s)\times\mathcal{O}}<\infty.$

 \item[(A5) Momentum conservation: ]
 Denote $n\in \mathbb{Z}^2$ by $((n)_1,(n)_2) $. The perturbation $P$ satisfies $[P, \mathbb{M}^{( l)}]=0,\,\,(l=1,2),$ where
\begin{equation*}
    \mathbb{M}^{( l)} = \sum^d_{h=1}( \sum^b_{a=1}(i^{(a)})_l\frac{\partial}{\partial \theta_{ha}}+ \sum_{\varrho=\pm} \sum_{n\in \mathbb{Z}^2_1}\varrho \mi (n)_lz^{\varrho}_{hn}\frac{\partial}{\partial z^{\varrho}_{hn}} )
\end{equation*}

\item[(A6) T\"{o}plitz-Lipschitz property: ]
Let
$ \Lambda=\sum^d_{h=1}\sum\limits_{\sigma=\pm}\sigma\mi(\sum\limits_{n\in \mathbb{Z}^2_1}\Omega^0_{hn}(\xi) z^\sigma_{hn}\frac{\partial}{\partial z^\sigma_{hn}}
  ).$
For  fixed $n,m\in \mathbb{Z}^2$ $c\in \mathbb{Z}^2\setminus\{0\}, $ the following limits exist and satisfy:
\begin{equation}\label{tl0}
     \|\lim_{t\rightarrow\infty}\frac{\partial P^{(x)}}{\partial z^\sigma_{h'(n+tc)} }\|_{D_\rho(r,s)\times\mathcal{O}}\leq \varepsilon_0, \,\,\,\,x=\theta_{ha}, I_{ha}.
\end{equation}
\begin{equation}\label{tl1}
\|\lim_{t\rightarrow\infty}\frac{\partial (\Lambda+\mathcal{A}+P)^{(z^\sigma_{h (n+tc)})}}{\partial z^{\pm\sigma}_{h(m\pm tc)}}\|_{D_\rho(r,s)\times\mathcal{O}}
\leq \varepsilon_0\me^{-|n\mp m|\rho}.
\end{equation}
\begin{equation}\label{tl2}
	\|\lim_{t\rightarrow\infty}\frac{\partial (P)^{(z^\sigma_{h (n+tc)})}}{\partial z^{\pm\sigma}_{h'(m\pm tc)}}\|_{D_\rho(r,s)\times\mathcal{O}}
	\leq \varepsilon_0\me^{-|n\mp m|\rho}.
\end{equation}
Furthermore, there exists $K>0$  such that when  $|t|>K,$ the following estimates hold.
\begin{equation}\label{tl3}
   \|\frac{\partial P^{(x)}}{\partial z^\sigma_{h' (n+tc)} }-\lim_{t\rightarrow\infty}\frac{\partial P^{(x)}}{\partial z^\sigma_{h' (n+tc)} }\|_{D_\rho(r,s)\times\mathcal{O}}\leq \varepsilon_0,
\end{equation}
$x=\theta_{ha}, I_{ha}.$
\begin{equation}\label{tl4}
\|\frac{\partial (\Lambda+\mathcal{A}+P)^{(z^{\sigma}_{h (n+tc)})}}{\partial z^{\pm\sigma}_{h'(m\pm tc)}}-\lim_{t\rightarrow\infty}\frac{\partial (\Lambda+\mathcal{A}+P)^{(z^{\sigma}_{h (n+tc)})}}{\partial z^{\pm\sigma}_{h (m\pm tc)}}\|_{D_\rho(r,s)\times\mathcal{O}}\leq \frac{\varepsilon_0}{|t|}\me^{-|n\mp m|\rho}.
\end{equation}
\begin{equation}\label{tl5}
	\|\frac{\partial ( P)^{(z^{\sigma}_{h (n+tc)})}}{\partial z^{\pm\sigma}_{h'(m\pm tc)}}-\lim_{t\rightarrow\infty}\frac{\partial ( P)^{(z^{\sigma}_{h (n+tc)})}}{\partial z^{\pm\sigma}_{h'(m\pm tc)}}\|_{D_\rho(r,s)\times\mathcal{O}}\leq \frac{\varepsilon_0}{|t|}\me^{-|n\mp m|\rho}.
\end{equation}
 \end{description}

\begin{remark}
In (A6), the conditions \eqref{tl1} and \eqref{tl4} are the most important  for measure estimates. The
role played by the conditions \eqref{tl0} and \eqref{tl3}  is to preserve T\"{o}plitz-Lipschitz property   after the KAM iteration.
\end{remark}

Now we state our KAM theorem.
\begin{theorem}\label{KAM}
Suppose the $S-$reversible vector field $X=N+\mathcal{A }+P$ in \eqref{perbVF} satisfies $(A1)-(A6).$
Let $\gamma>0$ be small enough, then there exists a positive $\varepsilon$ depending  only on $d,b, L, K, \tau, r, s$
and $\rho$ such that if $\|P\|_{s;D(r,s)\times\mathcal{O}}\leq\varepsilon,$   the following holds: There exist

(1)
 a Cantor subset $\mathcal{O}_\gamma\subset\mathcal{O}$ with Lebesgue measure $\meas(\mathcal{O}\setminus\mathcal{O}_\gamma)=O(\gamma^{1/4});$

(2) a $C^{4}_W-$smooth family of  real analytic  torus embeddings
$$\Psi:\mathbb{T}^{db}\times\mathcal{O}_\gamma\rightarrow D_\rho(r,s)$$
which is $\frac{\varepsilon}{\gamma^{4}}-$close to the trivial embedding
$\Psi_0:\mathbb{T}^{db}\times\mathcal{O}\rightarrow \mathcal{T}^{db}_0;$

(3) a $C^{4}_W-$smooth  map $\phi:\mathcal{O}_\gamma\rightarrow \mathbb{R}^{db}$
which is  $\varepsilon-$close to the unperturbed frequency $\omega$
such that for every $\xi\in \mathcal{O}_\gamma$ and $\theta\in\mathbb{T}^{db}$
the curve $t\mapsto \Psi(\theta+\phi(\xi)t; \xi)$ is a quasi-periodic solution of the equation governed by
the  vector field $X=N+\mathcal{A}+P.$
\end{theorem}

Due to the (A2*) condition, the resonance caused by the coupling of equations disappears.
Thus the proof of the KAM step is   simpler than \cite{Geng21} and similar to \cite{Geng11}.  So we only  verify our equation satisfying (A1)-(A6).

\section[]{Application to the Coupled NLS}
\hspace{1.5em}
Let tangential sites $\mathcal{I}=\{i^{(1)}, i^{(2)},\cdots,i^{(b)}\}\subset \mathbb{Z}^2$.
Under periodic boundary conditions, we denote   the
eigenvalues of  $\Delta$ by $\lambda_n,\,n\in \mathbb{Z}^2$
$$\omega_{hi^{(a)}}=\lambda_{i^{(a)}}=|i^{(a)}|^2,\,\,1\leq a\leq b,\ \Omega_{hn}=\lambda_{n}=|n|^2 ,\,\,n\in \mathbb{Z}^2_1,$$
and the corresponding eigenfunctions $\phi_n(x)=(2\pi)^{-1}\me^{\mi \langle n,x\rangle}$

Scaling $u_h\rightarrow \varepsilon^\frac{1}{2}u_h$. Let $u_h(t,x)=\sum\limits_{n\in \mathbb{Z}^2}q_{hn}(t)\phi_n(x), h=1,\cdots,d$
then we obtain the equivalent lattice reversible equations respect to $S(q,\bar{q} )=(\bar{q},q),$
 \begin{equation}\label{CNLS2}
                       \begin{cases}
                        \dot{q}_{hn}=\mi\lambda_n q_{hn}+\varepsilon Q^{(q_{hn})}(q,\bar{q}),\quad                                 \\
                        \dot{\bar{q}}_{hn}=-\mi\lambda_n \bar{q}_{hn}+\varepsilon Q^{(\bar{q}_{hn})}(q,\bar{q}),                     \,n\in \mathbb{Z}^2,\, h=1,\cdots,d.
                       \end{cases}
\end{equation}
which corresponding an $S$-reversible vector field
\begin{equation}\label{crrev}
  \tilde{X}=\sum_{h=1}^d\sum_{n\in \mathbb{Z}^2} (\mi\lambda_{n}q_{hn}\partial_{q_{hn}}-\mi\lambda_n\bar q_{hn}\partial_{\bar q_{hn}} +\varepsilon Q^{(q_{hn})}\partial_{q_{hn}}+\varepsilon Q^{(\bar{q}_{hn})}\partial_{\bar q_{hn}}).
\end{equation}
For $G_h=|u_h|^4+O(|u|^6),\, h=1,\cdots,d$, we have
\begin{equation}\label{gn1}
Q^{(q_{hn})}=\sum\limits_{i,j,m\in \mathbb{Z}^2} Q^{(q_{hn})}_{ijm}q_{hi}q_{hj}\bar{q}_{hm}+\sum_{|\alpha+\beta|\geq5} Q^{(q_{hn})}_{\alpha \beta  }q^\alpha \bar{q}^\beta ,Q^{(\bar{q}_{hn})}=\overline{Q^{(q_{hn})}}
\end{equation}
with
\begin{equation}\label{gn3}
  \begin{split}
 Q^{(q_{hn})}_{ijm}=&\mi \int_{\mathbb{T}^2}\phi_i\phi_j\bar{\phi}_m\bar{\phi}_n dx\\
                  =&
                       \begin{cases}
                        \frac{\mi}{(2\pi)^{2}} ,\quad                                      & i+j-m-n=0,\\
                        0                         \quad                   &  i+j-m-n\neq 0,
                       \end{cases}
 \\
 Q^{(q_{hn})}_{\alpha\beta}
                  =& 0, \quad \sum_{h=1}^d \sum_{m\in \mathbb{Z}^2}(\alpha_{hm}-\beta_{hm})m-n\neq0
\end{split}
\end{equation}

By direct computation, one can verify that
the perturbations $Q^{(q)}
$ satisfies the following regularity property.
\begin{lemma}\label{reg}
For any fixed $\rho > 0$, $Q^{(q)}$   is real analytic as a map in a neighborhood
of the origin with
$$\|Q^{(q)}\|_{\rho}\leq c \|q\|^3_{\rho}.$$
\end{lemma}

Let $P^0=\sum\limits_{\varrho=\pm}(Q^{(q^\varrho)}\frac{\partial}{\partial q^\varrho}),$ then we have
the following lemma.
\begin{lemma}\label{A56}
\emph{(1)} $[P^0, \mathbb{M}^{(l)}_0]=0,\,l=1,2,$
where
 \begin{equation}\label{M0}
  \mathbb{M}^{(l)}_0= \sum_{h=1}^d\sum_{\varrho=\pm} \sum_{n\in \mathbb{Z}^2}\varrho \mi (n)_lq^{\varrho}_{hn}\frac{\partial}{\partial q^{\varrho}_{hn}}
\end{equation}

\emph{(2)} $P^0$ satisfies T\"{o}plitz-Lipschitz property.
\end{lemma}

\begin{proposition}
   Let $\mathcal{I}$ be admissible. For vector field $\tilde{X}$ in \eqref{crrev}, there is a transformation $\Psi$ such that
\begin{equation}\label{Psi}
X=\Psi^* \tilde{X}:=(D\Psi)^{-1}\tilde{X}\circ \Psi=N+\mathcal{A}+P=\sum_{h=1}^d (N_h+\mathcal{A}_h)+P
\end{equation}
is a reversible system respect to
$S(\theta,I, z,\bar{z})=(-\theta, I, \bar{z}, z)$, where
\begin{equation*}
\begin{split}
  N_h
  =&\sum_{i\in \mathcal{I}}\omega_{hi}\frac{\partial}{\partial \theta_{hi}}+\sum_{\varrho=\pm1}\sum_{n\in \mathbb{Z}^2_1} \mi \varrho \Omega_{hn} z_{hn}^\varrho\frac{\partial}{\partial z_{hn}^\varrho}
\\
   +&\sum_{\varrho=\pm1}\varrho(\sum_{n\in \mathcal{L}_1}(\mi\omega_{hi} z_{hn}^\varrho\frac{\partial}{\partial z_{hn}^\varrho}
  ) - \sum_{n\in \mathcal{L}_2}(\mi\omega_{hi} z_{hn}^\varrho\frac{\partial}{\partial z_{hn}^\varrho}
  ))
 \end{split}
\end{equation*}
\begin{equation*}
\left\{\begin{aligned} \omega_{hi} &=\varepsilon^{-4} |i|^2  -\frac{1}{4\pi^2}\xi_{hi}+\sum_{j\in \mathcal{I}} \frac{1}{2\pi^2} \xi_{hj} \\ \Omega_{hn}&=\varepsilon^{-4} |n|^2   +\sum _{i\in \mathcal{I}} \frac{1}{2\pi^2} \xi_{hi} ,\quad \xi_h\in [h-\frac{1}{2},h]^b
  \end{aligned} \right.
\end{equation*}
\begin{equation*}
  \mathcal{A}_h=\sum_{\varrho=\pm1}\sum_{n\in \mathcal{L}_1}\mi \varrho\sqrt{\xi_{hi}\xi_{hj}} z_{hn}^\varrho \frac{\partial}{\partial z_{hm}^\varrho} +\sum_{\varrho=\pm1} \sum_{n\in \mathcal{L}_2} \mi\varrho \sqrt{\xi_{hi}\xi_{hj}}  \bar{z}_{hn}^\varrho\frac{\partial}{\partial z_{hm}^\varrho}
\end{equation*}
\begin{equation}\label{normalform}
 \| P\|=O(\varepsilon \|z\|)
\end{equation}
satisfies (A1)-(A6).
\end{proposition}
\begin{proof}
	
 The proof consists of several transformation. Firstly let
\begin{equation*}
	\begin{split}
		F=F^{(q)}\frac{\partial}{\partial q}+F^{(\bar{q})}\frac{\partial}{\partial p} {=\sum\limits_{{i-j+n-m=0,\atop |i|^2-|j|^2+|n|^2-|m|^2\neq0,}\atop \mathcal{I} \bigcap \{i,j,n,m\}\geq2,h=1,\cdots ,d} \frac{\mi \varepsilon}{4\pi^2(\lambda_i-\lambda_j+\lambda_n-\lambda_m)} q_{hi}\bar{q}_{hj} q_{hn}\frac{\partial}{\partial q_{hm}}} \\
		{+\sum\limits_{{i-j+n-m=0,\atop |i|^2-|j|^2+|n|^2-|m|^2\neq0,}\atop \mathcal{I} \bigcap \{i,j,n,m\}\geq2,h=1,\cdots, d} \frac{\mi \varepsilon}{4\pi^2(\lambda_i-\lambda_j+\lambda_n-\lambda_m)} q_{hi}\bar{q}_{hj}
			\bar{q}_{hm}\frac{\partial}{\partial \bar{q}_{hn}} }.
	\end{split}
\end{equation*}
\begin{equation*}
	\begin{split}
		(\phi^1_F)^* \tilde{X}&=  \sum_{\varrho=\pm 1,\,h=1,\cdots ,d}\varrho (\mi \sum_{i\in \mathcal{I}} \lambda_i q_{hi}^\varrho \frac{\partial}{\partial q_{hi}^\varrho }+\mi \sum_{i\in \mathbb{Z}^2_1} \lambda_i z_{hi}^\varrho  \frac{\partial}{\partial z_{hi}^\varrho }
		+  \frac{\mi \varepsilon}{4\pi^2} \sum_{i\in \mathcal{I}}  |q_{hi}|^2q_{hi}^\varrho  \frac{\partial}{\partial q_{hi}^\varrho } \\
	 &	+\frac{\mi \varepsilon}{2\pi^2}\sum_{i,j\in \mathcal{I},i\neq j}  |q_{hj}|^2q_{hi}^\varrho  \frac{\partial}{\partial q_{hi}^\varrho }
		+  \frac{\mi \varepsilon}{2\pi^2} \sum _{i\in \mathbb{Z}^2_1,j\in\mathcal{I}} |q_{hj}|^2z_{hi}^\varrho  \frac{\partial}{\partial z_{hi}^\varrho } \\
		&	+\frac{\mi \varepsilon}{2\pi^2}\sum_{n\in \mathcal{L}_1} \bar{q}_{hi}^\varrho q_{hj}^\varrho z_{hm}^\varrho \frac{\partial}{\partial z_{hn}^\varrho }
 	+\frac{\mi \varepsilon}{2\pi^2}\sum_{n\in \mathcal{L}_2}  q_{hi}^\varrho q_{hj}^\varrho \bar{z}_{hm}^\varrho  \frac{\partial}{\partial z_{hn}^\varrho })
		+O(\sum_{\varrho=\pm 1} (  \varepsilon\|z\|^3\frac{\partial}{\partial z^\varrho}  )).
	\end{split}
\end{equation*}
We remind that $(n,m)$ are resonant pairs,and$(i,j) $ is uniquely determined by  $(n,m)$.
We introduce  action-angle variables $(\theta, I)\in \mathbb{T}^{db}\times \mathbb{R}^{db}$
by
$$q_{hj}=\sqrt{I_{hj}+\xi_{hj}}\me^{\mi \theta_{hj}},\,\,\bar{q}_{hj}=\sqrt{I_{hj}+\xi_{hj}}\me^{-\mi \theta_{hj}},\,\,j\in \mathcal{I},$$
and
 scale in time
$  \xi\rightarrow \varepsilon^3 \xi ,I\rightarrow \varepsilon ^5 I, (z,\bar{z})\rightarrow \varepsilon ^\frac{5}{2}(z,\bar{z}). $
Choosing $\xi$ satisfying
$$
    \xi_h\in [h-\frac{1}{2},h]^b,\  h=1,\cdots, d,
$$
we finally get
\begin{equation}
	X=\varepsilon^{-9}X(\varepsilon^3\xi,\varepsilon^5I,\theta,\varepsilon ^\frac{5}{2}z,\varepsilon ^\frac{5}{2} \bar{z})=N+ \mathcal{A}+P=\sum_{h=1}^d (N_h+\mathcal{A}_h)+P
\end{equation}
where
\begin{equation*}
	\begin{split}
		N_h
		=&\sum_{i\in \mathcal{I}}\omega_{hi}\frac{\partial}{\partial \theta_{hi}}+\sum_{\varrho=\pm1}\sum_{n\in \mathbb{Z}^2_1} \mi \varrho \Omega_{hn} z_{hn}^\varrho\frac{\partial}{\partial z_{hn}^\varrho}
		\\
		+&\sum_{\varrho=\pm1}\varrho(\sum_{n\in \mathcal{L}_1}(\mi\omega_{hi} z_{hn}^\varrho\frac{\partial}{\partial z_{hn}^\varrho}
		) - \sum_{n\in \mathcal{L}_2}(\mi\omega_{hi} z_{hn}^\varrho\frac{\partial}{\partial z_{hn}^\varrho}
		))
	\end{split}
\end{equation*}
\begin{equation*}
	\left\{\begin{aligned} \omega_{hi} &=\varepsilon^{-4}|i|^2-\frac{1}{4\pi^2}\xi_{hi}+\sum_{j\in\mathcal{I}} \frac{1}{2\pi^2} \xi_{hj} \\ \Omega_{hn}&=\varepsilon^{-4}|n|^2 +\sum _{i\in \mathcal{I}} \frac{1}{2\pi^2} \xi_{hi} ,\quad \xi_h\in [h-\frac{1}{2},h]^b
	\end{aligned} \right.
\end{equation*}
\begin{equation*}
	\mathcal{A}_h=\sum_{\varrho=\pm1}\sum_{n\in \mathcal{L}_1}\mi \varrho\sqrt{\xi_{hi}\xi_{hj}} z_{hn}^\varrho \frac{\partial}{\partial z_{hm}^\varrho} +\sum_{\varrho=\pm1} \sum_{n\in \mathcal{L}_2} \mi\varrho \sqrt{\xi_{hi}\xi_{hj}}  \bar{z}_{hn}^\varrho\frac{\partial}{\partial z_{hm}^\varrho}
\end{equation*}
\begin{equation*}
 \| P\|=O(\varepsilon \|z\|).
\end{equation*}
Then we give the verification of  assumptions (A1)-(A6) for \eqref{normalform}.

\emph{Verifying }(A1): It is obvious as the Jacobian matrix
$$\frac{\partial \omega_h}
 {\partial\xi_h}=\frac{1}{4\pi^2}\left( \begin{matrix} 1 &2 &...&2 \\2&1&...&2\\...&...&...&...\\2&2&...&1        \end{matrix} \right).$$

\emph{Verifying }(A2) and (A2*): It is obvious  $|\Omega_{hn}-\Omega_{h'n}|>\frac{b}{4\pi^2}.$

\emph{Verifying }(A3): For \eqref{normalform},$M_n$ read as follows:
$ M_{hn}=\Omega_{hn},\, n\in \mathbb{Z}^2_2, $
\begin{equation*}
\begin{split}
&M_{hn}=\left(\begin{matrix} \Omega_n+\omega_i&\frac{1}{2\pi^2} \sqrt{\xi_i\xi_j} \\
                     \frac{1}{2\pi^2} \sqrt{\xi_i\xi_j}&\Omega_m+\omega_j
                           \end{matrix}\right),\,n \in\mathcal{L}_1, \\
&M_{hn}=\left(\begin{matrix} \Omega_{n} -\omega_i&\frac{1}{2\pi^2} \sqrt{\xi_i\xi_j} \\
                     -\frac{1}{2\pi^2} \sqrt{\xi_i\xi_j}&-\Omega_{m} +\omega_j
                            \end{matrix}\right) ,\,n \in\mathcal{L}_2 ,\quad h=1\cdots d. \end{split}
\end{equation*}
We can  refer to  Section 3.2 in \cite{Geng11} since the proof is similar.

\emph{Verifying }(A4): Suppose   vector field \eqref{normalform} is defined
on the domain $D(r,s)$ with  $0<r<1, s=\varepsilon^2.$ It follows from \eqref{normalform} that
$ \|P\|_{s;D_\rho(r,s)\times\mathcal{O}}\leq c\varepsilon s\leq c\varepsilon.$

\emph{Verifying }(A5) and (A6): Through the transformation $\Psi$ in \eqref{Psi}, the vector fields  $\mathbb{M}^{(1)}_0$ and $\mathbb{M}^{(2)}_0$
in \eqref{M0} are transformed into $\mathbb{M}^{(l)}=\Psi^*\mathbb{M}^{(l)}_0$, then for $\, l\in\{1,2\}$
$$[P, \mathbb{M}^{(l)}]=[\Psi^*P^0, \Psi^*\mathbb{M}^{(l)}_0]=\Psi^*[P^0, \mathbb{M}^{(l)}_0]=0.$$

Lemma 5.4  in  \cite{Geng21} shows that the transformation $\Psi$  preserves the    T\"{o}plitz-Lipschitz property.
Thus $N+\mathcal{A}+P $ satisfies (A6).

\end{proof}

Using this proposition and Theorem 3.1, we can get Theorem 1.1.


\begin{thebibliography}{10}

\bibitem{Berti12}
M.~Berti and P.~Bolle.
\newblock Sobolev quasi-periodic solutions of multidimensional wave equations
with a multiplicative potential.
\newblock {\em Nonlinearity}, 25(9):2579--2613, 2012.

\bibitem{Berti13b}
M.~Berti and P.~Bolle.
\newblock Quasi-periodic solutions with {S}obolev regularity of {NLS} on {$\Bbb
	T^d$} with a multiplicative potential.
\newblock {\em J. Eur. Math. Soc.}, 15(1):229--286, 2013.



\bibitem{Bourgain94}
J.~Bourgain.
\newblock Construction of quasi-periodic solutions for {H}amiltonian
  perturbations of linear equations and applications to nonlinear {PDE}.
\newblock {\em Internat. Math. Res. Notices}, (11):475ff., approx. 21~pp.\,
  1994.

\bibitem{Bourgain98}
J.~Bourgain.
\newblock Quasi-periodic solutions of {H}amiltonian perturbations of 2{D}
  linear {S}chr\"odinger equations.
\newblock {\em Ann. of Math. (2)}, 148(2):363--439, 1998.

\bibitem{Bourgain05}
J.~Bourgain.
\newblock {\em Green's function estimates for lattice {S}chr\"odinger operators
  and applications}, volume 158 of {\em Annals of Mathematics Studies}.
\newblock Princeton University Press, Princeton, NJ, 2005.



\bibitem{Chierchia00}
L.~Chierchia and J.~You.
\newblock K{AM} tori for 1{D} nonlinear wave equations with periodic boundary
  conditions.
\newblock {\em Comm. Math. Phys.}, 211(2):497--525, 2000.


\bibitem{Eliasson10}
L.~Eliasson and S.~Kuksin.
\newblock K{AM} for the nonlinear {S}chr\"odinger equation.
\newblock {\em Ann. of Math. (2)}, 172(1):371--435, 2010.

\bibitem{Gao11}
 M.~Gao  and J.~Zhang.
\newblock   Small-divisor equation of higher
	order with large variable coefficient and application to the coupled KdV equation.
\newblock {\em Acta Mathematica Sinica English}, 27(10) :2005-2032, 2011.



\bibitem{Geng11}
J.~Geng, X.~Xu, and J.~You.
\newblock An infinite dimensional {KAM} theorem and its application to the two
  dimensional cubic {S}chr\"odinger equation.
\newblock {\em Adv. Math.}, 226(6):5361--5402, 2011.


\bibitem{Geng13}
J.~Geng and J.~You.
\newblock A {KAM} theorem for higher dimensional nonlinear {S}chr\"odinger
  equations.
\newblock {\em J. Dynam. Differential Equations}, 25(2):451--476, 2013.

\bibitem{Geng21}
J.~Geng,\, Z.~Lou and Y.~Sun.
\newblock  A KAM Theorem for Two Dimensional Completely Resonant
Reversible Schr\"odinger Systems
\newblock {\em J. Dynam. Differential Equations}, 2021.

\bibitem{Grebert}
B.~Gr\'ebert and V.~Vila\c{c}a Da Rocha.
\newblock Stable and unstable time quasi periodic solutions for a system of coupled NLS equations.
\newblock {\em Nonlinearity}, 31, 4776-4811, 2018.


\bibitem{Kuksin87}
S.~Kuksin.
\newblock Hamiltonian perturbations of infinite-dimensional linear systems with
  imaginary spectrum.
\newblock {\em Funktsional. Anal. i Prilozhen.}, 21(3):22--37, 95, 1987.

\bibitem{Kuksin96}
S.~Kuksin and J.~P{\"o}schel.
\newblock Invariant {C}antor manifolds of quasi-periodic oscillations for a
  nonlinear {S}chr\"odinger equation.
\newblock {\em Ann. of Math. (2)}, 143(1):149--179, 1996.




\bibitem{Poschel96a}
J.~P{\"o}schel.
\newblock A {KAM}-theorem for some nonlinear partial differential equations.
\newblock {\em Ann. Scuola Norm. Sup. Pisa Cl. Sci. (4)}, 23(1):119--148, 1996.

\bibitem{Procesi15}
C.~Procesi and M.~Procesi.
\newblock A {KAM} algorithm for the resonant non-linear {S}chr\"odinger
  equation.
\newblock {\em Adv. Math.}, 272:399--470, 2015.

\bibitem{Shi}
 Y.~Shi and J.~Xu.
\newblock  Quasi-periodic solutions for a class of beam equation system.
\newblock  {\em Discret. Cont. Dyn. B},  25(1): 31-53, 2020.






\bibitem{WangW16}
W.~Wang.
\newblock Energy supercritical nonlinear {S}chr\"odinger equations:
  quasiperiodic solutions.
\newblock {\em Duke Math. J.}, 165(6):1129--1192, 2016.


\bibitem{Wayne90}
C.~Wayne.
\newblock Periodic and quasi-periodic solutions of nonlinear wave equations via
  {KAM} theory.
\newblock {\em Comm. Math. Phys.}, 127(3):479--528, 1990.

\bibitem{X}
S.~Xue.
\newblock A KAM algorithm for two-dimensional nonlinear Schr\"odinger equations with spatial variable.
 \newblock {\em J.Differential Equations},
 364: 1-52,2023.







\end{thebibliography}
\bibliographystyle{plain}

\end{document}